\documentclass[11pt]{amsart}
\usepackage{amssymb}
\usepackage{amscd, tikz-cd}
\usepackage[alphabetic]{amsrefs}

\usepackage[left=1.4in, right=1.4in, top=1.6in, bottom=1.6in]{geometry}
\usepackage[colorlinks, citecolor = blue]{hyperref}

\newtheorem{prop}{Proposition}
\newtheorem{theorem}[prop]{Theorem}
\newtheorem{lemma}[prop]{Lemma}

\newtheorem{conj}[prop]{Conjecture}
\newtheorem*{keyobs}{Key Observation}
\newtheorem*{fact}{Fact}

\newcommand{\Q}{\mathbb Q}

\newcommand{\A}{\mathbb A}

\newcommand{\Sym}{\mathrm{Sym}}

\newcommand{\Ree}{\mathrm{Re}}
\newcommand{\den}{\mathrm{den}}
\newcommand{\Fr}{\mathrm{Fr}}

\DeclareMathOperator{\GL}{GL}
\DeclareMathOperator{\Gal}{Gal}
\DeclareMathOperator{\tr}{tr}
\DeclareMathOperator{\Ad}{Ad}

\begin{document}

\title[Strong local--global phenomena for representations]{Strong local--global phenomena for Galois and automorphic representations}
\date{\today}

\author{Kimball Martin}

\address{Department of Mathematics, University of Oklahoma, Norman, OK 73019 USA}


\begin{abstract}
Many results are known regarding how much local information is required to determine
a global object, such as a modular form, or a Galois or  automorphic representation.
We begin by surveying some things that are known and expected, and then explain recent
joint work with Dinakar Ramakrishnan  about comparing degree 2 Artin and automorphic
representations which {\em a priori} may not correspond at certain infinite sets of places.
\end{abstract}

\maketitle

These notes are based on a talk I gave at the RIMS workshop, ``Modular forms and
automorphic representations,'' from Feb 2--6, 2015, which was in turn based
on the joint work \cite{MR} with Ramakrishnan.  I am grateful to the organizers
for the opportunity to present this exposition.  

These notes were written while I was visiting
Osaka City University under a JSPS Invitation Fellowship.  I was also supported in part by
a Simons Collaboration Grant.  I am happy to thank all of these organizations for their kind support.
I would also like to thank Christina Durfee for enlightening discussions about characters of finite
groups and Nahid Walji for helpful feedback.

\section{Introduction}
A local--global principle, or phenomenon, is a situation where certain
local conditions are sufficient to imply a corresponding global condition.  
Examples both of local--global principles (e.g., zeroes of quadratic forms, norms in cyclic
extensions, {\em Grunwald}--Wang, splitting of central simple algebras) 
as well as 
examples of failures of local--global principles (e.g., unique factorization, Grunwald--\emph{Wang},
points on varieties, zeros or poles of $L$-functions, 
vanishing of periods) abound in number theory and are of consummate interest.
See, for example, Mazur's (8th out of 11 so far!) {\em Bulletin} article \cite{mazur:1993} for 
local--global principles and obstructions for varieties.

On the other hand, for certain objects like idele class characters,
modular (new) forms, Galois representations or 
automorphic representations, we have much more rigid local--global phenomena.
Here the usual local--global principle is more-or-less tantamount to the existence of an Euler product
for the associated $L$-function.  We will discuss stronger versions of this, where
knowing local $L$-factors at a sufficiently large set of places determines the global
$L$-function (and hence, often, the global object up to isomorphism).

Specifically, consider the following 3 results.  Let $F$ be a number field, $\Sigma_F$
 the set of places of $F$ and $\Gamma_F$ the absolute Galois group of $F$.  
Denote by $\rho, \rho'$ irreducible $n$-dimensional complex representations 
of $\Gamma_F$ (i.e., irreducible Artin representations)
and by $\pi, \pi'$ irreducible cuspidal automorphic representations of $\GL_n(\A_F)$.

\begin{enumerate}
\item If $L(s, \rho_v) = L(s, \rho'_v)$ for almost all $v$, then $L(s, \rho) = L(s, \rho')$, and in fact
$\rho \simeq \rho'$.

\item If $L(s, \pi_v) = L(s, \pi'_v)$ for almost all $v$, then $L(s, \pi) = L(s, \pi')$,
and in fact $\pi \simeq \pi'$.

\item If $L(s, \rho_v) = L(s, \pi_v)$ for almost all $v$, then $L(s, \rho) = L(s, \pi)$.
\end{enumerate}

To be more precise, by the notation $L(s, \rho)$, $L(s, \pi)$, etc., for global $L$-functions we will
mean the incomplete $L$-function (the product over all finite places of local factors).  When we
want to denote completed $L$-functions, we will write $L^*(s, \rho)$, $L^*(s, \pi)$, etc.  
By equality of two global $L$-functions, we mean as Euler products over the base field, i.e., not
just equality of meromorphic functions but all local factors are equal as well.

Result (1) is an elementary consequence of Chebotarev density, and 
(2) is the strong multiplicity one (SMO) theorem for $\GL(n)$ due to Jacquet and Shalika \cite{JS}.  
Result (3) follows from an argument of Deligne and Serre \cite{DS} (see
Appendix A of my thesis \cite{martin}).

While the first two statements are usually stated just 
with the conclusion of the two representations being isomorphic, stating the conclusion in terms of a
global $L$-function equality puts all three results on the same footing.  In addition, if one wants to
think about representations of other groups, this seems to be the right point of view.
E.g., cuspidal representations of $\mathrm{SO}_n(\A)$ will not satisfy SMO in the usual sense, but equality
at almost all places should give an equality of global $L$-functions (in fact, $L$-packets).
\medskip

Now one can ask a more general type of question.  
Suppose two global $L$-functions over $F$ agree at all primes outside
of some set $S \subset \Sigma_F$.  Under what conditions can we conclude that the $L$-factors
are equal everywhere?
The above 3 results are about when $S$ is a finite set, but some results and conjectures exist
generalizing (1) and (2) if $S$ is ``not too big'', or of a certain form.  We will discuss each of
these situations, and conclude by explaining recent joint work with
Ramakrishnan \cite{MR}, where we generalized (3) to certain kinds of infinite sets for $n=2$.

Of course it is interesting to consider when $\rho$ and $\rho'$ are $\ell$-adic Galois representations
as well.  We will make some remarks about $\ell$-adic representations, but for simplicity focus
on Artin representations.

\section{Galois representations}

Suppose $\rho$ and $\rho'$ are irreducible $n$-dimensional Artin representations of $\Gamma_F$.
They both factor through Galois groups of finite extensions of $F$, so we can choose a single 
finite Galois extension $K/F$ such that $\rho$ and $\rho'$ may be considered as representations of 
$G = \Gal(K/F)$.  

Our basic problem is to determine if knowing $L^S(s, \rho) = L^S(s, \rho')$ for some fixed $S \subset
\Sigma_F$ implies $L(s, \rho) = L(s, \rho')$.
Since a representation of a finite group is determined by its character, for Artin representations
it suffices to consider a weaker hypothesis.  

Namely, let $S \subset \Sigma_F$ and suppose
$\tr \rho(Fr_v) = \tr \rho'(Fr_v)$ for $v \not \in S$.  (We may assume $S$ contains all places where 
$\rho$ and $\rho'$ are ramified, so that this makes sense.)  Note this is weaker than the condition 
on $L$-factors because, at unramified places,
$L_v(s, \rho) = (\det(I-\rho(\Fr_v)q_v^{-s}))^{-1}$ determines $\tr \rho(\Fr_v)$ but not conversely.

Recall we define the (natural) density of $S$ to be
\[ \den(S) := \lim_{x \to \infty} \frac{ \# \{ v \in S : q_v < x \} } { \# \{ v \in \Sigma_F : q_v < x \} }, \]
if this limit exists.  Now Chebotarev density 
says that if $\den(S) < \frac 1{|G|}$, then
$\{ \Fr_v : v \not \in S \}$ hits all conjugacy classes in $G$.  So if $\den(S) < \frac 1{|G|}$, then
$\tr \rho(g) = \tr \rho'(g)$ for all $g \in G$, whence $\rho \simeq \rho'$.  Often it is easier
to work with Dirichlet density, which is defined by
\[ \delta(S) := \lim_{s\to 1^+} \frac{ \sum q_v^{-s} }{\log \frac 1{s-1}} \]
If $\den(S)$ exists, so does $\delta(S)$ and they are equal.  

\begin{prop} Suppose $\rho$ and $\rho'$ are $n$-dimensional Artin representations of 
$\Gal(K/F)$.  If $\tr \rho(\Fr_v) = \tr \rho'(\Fr_v)$ for $v$ outside of a set $S$ of places with
$\delta(S) < \frac 1{2n^2}$, then $\rho \simeq \rho'$.
\label{prop:gal-SMO}
\end{prop}

This follows from combining the above Chebotarev density argument and the following result about finite group characters.

\begin{lemma} If $\chi$ and $\chi'$ are irreducible characters of degree $n$ of a finite group $G$
and $X = \{ g \in G : \chi(g) = \chi'(g) \}$ has size $> |G|(1-1/2n^2)$, then $\chi = \chi'$.
\end{lemma}

\begin{proof} Put $Y = G - X$.  Since $\chi, \chi'$ have maximum absolute value $n$, we see
\[ \sum_{g \in Y} |\chi(g) \bar \chi(g)|, \sum_{g \in Y} |\chi(g) \bar \chi'(g)| \le |Y|n^2 < \frac{|G|}2. \]
Since $\sum_{g \in G} \chi(g) \bar \chi(g) = |G|$, this means $\sum_{g \in X} \chi(g) \bar \chi(g) \ge
\frac{G}2$.  Then
\[ \sum_{g \in G} \chi(g) \bar \chi(g)  = \sum_{g \in X} \chi(g) \bar \chi(g) + \sum_{g \in Y} \chi(g) \bar \chi'(g) \ne 0, \]
which implies $\chi = \chi'$ by orthogonality relations and irreducibility.
\end{proof}

It is known that one cannot do better than this, cf.\ \cite{ramakrishnan:motives}.  
Namely, if $n=2^m$ then Buzzard, Edixhoven
and Taylor constructed distinct $n$-dimensional irreducibles $\rho$ and $\rho'$ such that
$\den(S) = |G|(1-1/2n^2)$ where $G$ is a central quotient of $Q_8^m$ ($Q_8$ is the quaternion
group of order 8) times $\{ \pm 1 \}$.  Serre showed the existence of similar examples for arbitrary
$n$.

We remark that Rajan \cite{rajan:l-adic} proved an analogue for (semisimple, finitely ramified) 
$\ell$-adic Galois representations of $\Gal(\bar F/F)$.

\section{Automorphic representations}

Let $\pi$ and $\pi'$ be irreducible automorphic cuspidal unitary representations of $\GL_n(\A_F)$.

For a finite place $v$, we can write
\[ L(s, \pi_v) = \prod_{i=1}^k (1-\alpha_{v,i} q_v^{-s})^{-1} \]
for some $0 \le k \le n$ and nonzero complex numbers $\alpha_{v,i}$.
Note that $L(s, \pi_v)$ is a nowhere vanishing meromorphic function whose set of poles are
precisely the values of $s$ such that $q_v^{s} = \alpha_{v,i}$ for some $1 \le i \le k$.  The latter
condition implies $q_v^{\Ree(s)} = |\alpha_{v,i}|$, and conversely for each real $x$ with
$q_v^x = |\alpha_{v,i}|$ for some $i$, there exists an $s$ such that $\Ree(s) = x$ and there
is a pole at $s$.  (If $k=0$, then $L(s, \pi_v) = 1$ and there are no poles.)

Similarly, write 
\[ L(s, \pi_v') = \prod_{i=1}^{k'} (1-\alpha'_{v,i} q_v^{-s})^{-1}. \]
The following observation, while simple, will be key for us in several places, so I will set
it off to highlight it.

\begin{fact}  Fix a finite place $v$ with $k, k' \ge 1$.  If the first (rightmost) pole for $L(s, \pi_v)$
occurs on the vertical line $\Ree(s) = x_0$, then $x_0 = \max \{ \frac{\log |\alpha_{v, i} |}{\log q_v} : 1 \le i \le k \}$.
Similarly, if the first pole for $L(s, \pi_v \times \bar \pi'_v)$
occurs on the vertical line $\Ree(s) = x_0$, then $x_0 = \max \{ \frac{\log |\alpha_{v, i} \alpha'_{v,j} |}{\log q_v} : 1 \le i \le k,
1 \le j \le k' \}$.
\end{fact}

\begin{theorem}[Strong Multiplicity One, \cite{JS}]  Suppose $L(s, \pi_v) = L(s, \pi'_v)$ for all $v$ outside of a
finite set $S$.  Then $\pi \simeq \pi'$.
\label{thm:SMO}
\end{theorem}

This generalizes earlier results of Miyake \cite{miyake} for $n=2$ and Piatetski--Shapiro \cite{PS},
who needed to also assume the archimedean components match.
For $n=1$, this follows from strong approximation.

\begin{proof}
There are two ingredients, both proved in \cite{JS}.

(i) We have $\pi \simeq \pi'$ if and only if
$L(s, \pi \times \bar \pi')$ has a pole at $s=1$
(use the integral representation and orthogonality of cusp forms).

(ii) For finite $v$, we have the bound $|\alpha_{v,i}| < q_v^{1/2}$.  (The ramified case reduces to the unramified case.)

Now to prove strong multiplicity one, consider the ratio
\begin{equation} \label{eq:1}
 \frac{L(s, \pi \times \bar \pi)}{L(s, \pi \times \bar \pi')} =
\frac{L_S(s, \pi \times \bar \pi)}{L_S(s, \pi \times \bar \pi')}.
\end{equation}

By (ii), we see that
$L_S(s, \pi \times \bar \pi)$ and $L_S(s, \pi \times \bar \pi')$ both have no poles on $\Ree(s) \ge 1$.
Since these functions are also never zero, the right hand side has no pole in
$\Ree(s) \ge 1$.  
Since $L(s, \pi \times \bar \pi)$ has a pole at $s=1$ by (i), it must be canceled out
by a pole of $L(s, \pi \times \bar \pi')$ at $s=1$ in order for the ratio on the left to not have a pole 
there.  Thus, again by (i), we get $\pi \simeq \pi'$.
\end{proof}

We remark that Moreno \cite{moreno} proved an ``analytic'' SMO:  
if $\pi$ and $\pi'$ have bounded conductors and archimedean
parameters, there is an effective (but exponential) constant $X$ 
(depending on the bounds on conductors and archimedean 
parameters, $F$ and $n$) such that if $\pi_v \simeq \pi'_v$
for all $v$ with $q_v < X$, then $\pi \simeq \pi'$.  Note that for $n=2$, such a result gives you a
bound on the number of Fourier coefficients needed to distinguish modular forms of bounded
level and weight with the same nebentypus.  (Of course there will be some finite bound
because the space of such forms is finite dimensional.)
Further work has been done along these lines (for $n=2$ and general $n$),
but this is not our focus now and we will not discuss it further.  We are interested in results where
one does not impose {\em a priori} bounds on ramification or infinity types.

\medskip
Coming back to the usual SMO, note that in the above proof it was 
crucial $S$ be finite to conclude the RHS of \eqref{eq:1}
has no pole in $\Ree(s) \ge 1$.
To refine this, we need a couple more ingredients.  

First is an improvement on (ii).  Recall the Generalized Ramanujan
Conjecture (GRC) asserts that each $\pi_v$ is tempered, i.e., each $|\alpha_{v,i}| = 1$.
For general $n$, the best that is known is the Luo--Rudnick--Sarnak bound from \cite{LRS},
which says $|\alpha_{v,i}| < q_v^{1/2-1/(n^2+1)}$.  For $n=2$ we can do better.  Using $\Sym^2$,
Gelbart--Jacquet \cite{GJ} got a bound of  $q_v^{1/4}$.  With $\Sym^3$ this was improved to 
exponent $\frac 19$ by Kim--Shahidi \cite{KSh}, then further improved to $\frac 7{64}$ by 
Kim--Sarnak \cite{KSa} and Blomer--Brumley \cite{BB} using $\Sym^4$.
In fact, for us, a bound of the form $q_v^{\delta}$ for some $\delta < \frac 14$ is sufficient.

The second ingredient we need is Landau's lemma, which we explain now.

Let us say a Dirichlet series $L(s)$
is of {\bf positive type} if it has an Euler product (on some right half plane) and $\log L(s)$ is a Dirichlet
series with positive ($\ge 0$) coefficients.  Note 
\[ \log \frac 1{1-\alpha q^{-s}} = \sum_n \frac{\alpha^n/n}{q^{ns}} \]
is a Dirichlet series with positive coefficients if $\alpha \ge 0$.  Since the sum of Dirichlet series
with positive coefficients is again a Dirichlet series with positive coefficients (admitting convergence)
we see an $L$-series of the form 
\[ L(s) = \prod_i \frac 1{1-\alpha_i q_i^{-s}} \]
with each $\alpha_i \ge 0$ is of positive type (admitting convergence), e.g., a Dedekind zeta 
function.   More important for us will be examples like $L(s, \pi \times \bar \pi)$, 
which are also of positive type.

\begin{lemma}[Landau] Suppose $L(s)$ is a Dirichlet series of positive type.  Then no zero of $L(s)$
occurs to the right of the first (rightmost) pole, and the first pole occurs on the real axis.
\end{lemma}

This will be extremely useful because we can now control the locations of not just poles of
$L$-functions, but also zeroes. 

\begin{theorem}[Refined SMO, Ramakrishnan \cite{ramakrishnan:SMO}] Suppose $n=2$ and $L(s, \pi_v) = L(s, \pi'_v)$ for all $v$
outside a set $S$ with $\delta(S) < \frac 18$.  Then $\pi \simeq \pi'$.
\label{thm:dinSMO}
\end{theorem}

This was used by Taylor \cite{taylor} for constructing families of $\ell$-adic Galois representations to
modular forms over imaginary quadratic field.

\begin{proof}[Proof (Sketch)]
Suppose $\pi \neq \pi'$ and assume $S$ contains all places of ramification.  Put
\[ Z(s) = \frac{L(s, \pi \times \bar \pi)L(s, \pi' \times \bar \pi')}{L(s, \pi \times \bar \pi') L(s, \pi' \times \bar \pi)}. \]
Then, by (i), the numerator has a double pole at $s=1$ while the denominator has no pole there.  
Hence $Z(s)$ has a pole of order 2 at $s=1$.  By definition, $Z_v(s) = 1$ for any $v \not \in S$,
so we also have
\[ Z(s) = Z_S(s) =  \frac{L_S(s, \pi \times \bar \pi)L_S(s, \pi' \times \bar \pi')}{L_S(s, \pi \times \bar \pi') L_S(s, \pi' \times \bar \pi)}. \]
Put
\[ D_S(s) = L_S(s, \pi \times \bar \pi)L_S(s, \pi' \times \bar \pi')L_S(s, \pi \times \bar \pi') 
L_S(s, \pi' \times \bar \pi) \]
so 
\[ Z_S(s) = \frac{L_S(s, \pi \times \bar \pi)^2 L_S(\pi' \times \bar \pi')^2}{D_S(s)}. \]
This is convenient because $D_S(s)$ is a Dirichlet series of positive type, 
and one can check it is nonvanishing for $s \ge 1$, so it has no zero at $s=1$ by Landau's lemma.
We would like to get a contradiction by saying that $L_S(s, \pi \times \bar \pi)$ and 
$L_S(\pi' \times \bar \pi')$ can't have poles at $s=1$ for $S$ of sufficiently small density, but there is no reason they even need to be meromorphic at $s=1$.

Instead, we observe that $Z_S(s)$ having a pole of order 2 at $s=1$ means
\[ \lim_{s \to 1^+} \frac{\log Z_S(s)}{\log \frac 1{s-1}} = 2. \]
Hence to obtain a contradiction, it will suffice to show 
\begin{equation} \label{eq:desired}
 \lim_{s\to 1^+} \frac{\log L_S(s, \pi \times \bar \pi)}{\log \frac 1{s-1}} < \frac 12, 
\end{equation}
as the same argument will apply to $L_S(s,\pi \times \bar \pi')$.
For simplicity, assume $F=\Q$.  Say $\pi_p$ has Satake parameters 
$\{ \alpha_{1,p}, \alpha_{2,p} \}$.
Then
\begin{align*}
 \log L(s, \pi_p \times \bar \pi_p) &=  \sum_{1 \le i, j \le 2 } \log \frac 1{1-\alpha_{i,p}\alpha_{j,p}p^{s}}
= \sum_{1 \le i, j \le 2 } \sum_{n \ge 1} \frac{(\alpha_{i,p}\alpha_{j,p})^n}{np^{ns}} \\
&= \frac{c_p} {p^{s}} + O(p^{-2s}),
\end{align*}
where
\[ c_p = \sum_{1 \le i, j \le 2 } \alpha_{i,p}\alpha_{j,p}. \]
It is well known that the ``prime zeta function'' satisfies
\[ \sum_p \frac 1{p^s} = \log \frac 1{s-1} + O(1), \quad s \to 1^{+}. \]
So if $\pi$ is tempered at $p$, and then $|c_p| \le 4$ and one deduces
\begin{equation} \label{eq:Lpipi-tb}
 \log L_S(s, \pi \times \bar \pi) \le 4\delta(S) \log \frac 1{s-1} + o(\log \frac 1{s-1}),
\end{equation}
and we are done as $\delta(S) < \frac 18$.

So the difficulty is when $\pi$ is not tempered.  Here one needs the above-mentioned bound
towards GRC: $|\alpha_{i,v}| \le q_v^{\delta}$ with $\delta < \frac 14$.
Then Ramakrishnan does a careful analysis involving $L(s, \Ad(\pi))$ and 
$L(s, \Ad(\pi) \times \Ad(\pi))$ to treat the case when $\Ad(\pi)$ is cuspidal.  
(This is the technical crux of the proof, but it will not come up later for us, so we will not explain
this analysis.)
If $\Ad(\pi)$ is not
cuspidal, then $\pi$ is induced from a character of a quadratic extension, and therefore tempered
everywhere.
\end{proof}

In fact, Rajan \cite{rajan:SMO} observed this is also true if one just assumes equality of
coefficients of Dirichlet series (i.e., sums of Satake parameters---or, for modular forms, Fourier
coefficients) at primes $v \not \in S$.  This is analogous to only requiring equalities of traces
$\tr \rho(\Fr_v) = \tr \rho'(\Fr_v)$ for Galois representations.

Note that Ramakrishnan's result is sharp, which one can deduce from $n=2$ examples which show
Proposition \ref{prop:gal-SMO} is sharp.  Nevertheless, Walji \cite{walji} was able to
prove some refinements, such as the following: 
if $n=2$ and $\pi$ and $\pi'$ are not dihedral (induced from quadratic extensions), 
then a refined SMO is true with the stronger bound $\delta(S) < \frac 14$.

\medskip
Now let's go back to considering arbitrary $n$.  

\begin{conj}[Ramakrishnan \cite{ramakrishnan:motives}] A refined SMO is true with $\delta(S) < \frac 1{2n^2}$.
\label{conj:din}
\end{conj}

For $n=1$ this is true by class field theory and Proposition \ref{prop:gal-SMO}.  For $n=2$, this
is precisely the content of Theorem \ref{thm:dinSMO}.

Let's think back to the proofs of Theorems \ref{thm:SMO} and \ref{thm:dinSMO} to see what
is needed to prove a refined SMO result.  In the proof of the usual SMO (Theorem \ref{thm:SMO})
we wanted to show $L_S(s, \pi \times \bar \pi)$ has no pole at $s=1$ and $L_S(s, \pi \times \bar \pi')$
has no zero at $s=1$.  To prove refined SMO for $\GL(2)$ (Theorem \ref{thm:dinSMO}), Ramakrishnan
considered a ratio $Z(s)$ and used Landau's lemma to essentially translate the problem
into showing both $L_S(s, \pi \times \bar \pi)$ and $L_S(s, \pi \times \bar \pi')$ have no poles in 
$\Ree(s) \ge 1$.  Let's just consider $L_S(s, \pi \times \bar \pi)$ since the idea for
$L_S(s, \pi \times \bar \pi')$ is similar.  

Suppose we have a bound towards GRC which says
each $L(s, \pi_v \times \bar \pi_v)$ has no pole in $\Ree(s) > 2\delta < 1$.  Then, morally, 
if $S$ is not too dense the bound for the first pole of $L_S(s, \pi \times \bar \pi)$ should not
be pushed too far to the right of $2\delta$.  (If $S$ has density 1, then $L_S(s, \pi \times \bar \pi)$
can have a pole up to 1 unit to the right of $2\delta$.)  The actual argument is more subtle than
this, but we will return to this moral shortly.

Looking at the argument for the tempered case of 
Theorem \ref{thm:dinSMO}, we see
the fact that $n=2$ was not really crucial.  For general $n$, the $4\delta(S)$ in \eqref{eq:Lpipi-tb}
becomes $n^2\delta(S)$, and this is less than the $\frac 12$ 
required in \eqref{eq:desired} precisely when $\delta(S) < \frac 1{2n^2}$.
In other words, this conjecture should follow from GRC and.  Moreover,
the bound $\delta(S) < \frac 1{2n^2}$ must be sharp by the existence of examples of
Galois representations showing Proposition \ref{prop:gal-SMO} is sharp---here one can take
these examples to be of finite nilpotent Galois groups, where one knows modularity by
Arthur--Clozel \cite{AC}.

Unfortunately, the Luo--Rudnick--Sarnak bounds toward GRC only tell us
each $L(s, \pi_v \times \bar \pi_v)$ has no pole in $\Ree(s) \ge 1-\frac{2}{n^2+1}$, which
does not seem to be enough to force $L_S(s, \pi \times \bar \pi)$ to have no pole in $\Ree(s) \ge 1$
for any $S$ of positive density.  So for not-necessarily tempered representations of $\GL(n)$
we don't know any refined SMO for $n > 2$ and $S$ of positive density at present, but
we can treat certain infinite sets $S$ of density 0.
(In fact, at the time of his conjecture, Ramakrishnan announced he had a weak result for $n > 2$ (\cite{ramakrishnan:SMO}, \cite{ramakrishnan:motives}), 
but did not publish a result of this type until recently---see below.)

We remark that there have been spectacular results on proving GRC for certain classes of
representations for $\GL(n)$ to which one often knows how to associate Galois representations,
e.g., cohomological self-dual representations over a totally real field.  For instance,
see Clozel's aphoristically titled article \cite{clozel}.

\medskip
In \cite{rajan:SMO}, Rajan showed that a refined SMO is true for arbitrary $n$ if
\[ \sum_{v \in S} q_v^{- \frac 2{n^2+1} } < \infty. \]
This is not difficult---this condition implies that the first pole of 
$L_S(s, \pi \times \bar \pi)$ is not more than $\frac 2{n^2+1}$ to the right of the first pole of a
local factor (the argument is the same as for the Key Observation below).  
So by the Luo--Rudnick--Sarnak bound, this is precisely what one needs to conclude
$L_S(s, \pi \times \bar \pi)$ has no pole in $\Ree(s) \ge 1$.
However this condition only holds for {\em very} sparse sets of primes.
It is much stronger than $\sum q_v^{-1} < \infty$, which is in turn stronger
than requiring $\delta(S) = 0$, so one cannot handle $S$ of positive density.
An example of where this applies is: let $F/\Q$ be cyclic of prime degree
$p > \frac{n^2+1}2$ and let $S \subset \Sigma_F$ consist of inert primes in $F/\Q$.  
(In \cite{rajan:SMO}, Rajan says $S$ has positive density in this example,
but presumably he means the corresponding primes of $\Q$, rather than $F$, have
positive density: by Chebotarev, the density of the underlying primes of $S$ in 
$\Sigma_\Q$ has density $\frac{p-1}p$ in $\Sigma_\Q$.)

\medskip
Recently, Ramakrishnan proved the following result.

\begin{theorem}[Ramakrishnan \cite{ramakrishnan}]
Suppose $F$ is a cyclic extension of prime degree $p$ of some number field $k$.  
A refined SMO is true when
$S \subset \Sigma_F$ contains only finitely many primes which are split over $k$.
\end{theorem}

This is still density 0, and satisfies Rajan's criterion when $p$ is large, so the main content is for
$p$ small.  In fact, $p=2$ is the hardest case, and this case was used in a crucial
way in trace formula comparisons of Wei Zhang \cite{zhang} and Feigon--Martin--Whitehouse
\cite{FMW}.
More recently, Ramakrishnan \cite{ramakrishnan:general} 
has extended this to the arbitrary Galois case, where a quite different 
approach was required.

\begin{proof}[Proof (sketch)]
As explained above, the key point is to show that $L_S(s, \pi \times \bar \pi)$ has no pole in 
$\Ree(s) \ge 1$.  There are two ingredients to the proof.  First is the following elementary but key
fact, which we want to highlight because we will use it again in the next section.  

\begin{keyobs} Let $S_j$ be the set of primes of degree $j$.  Suppose for each $v \in S_j$
we have numbers $\alpha_v$ such that $|\alpha_v| < q_v^\delta$.  Then $\prod_{S_j} \frac 1{1-\alpha_v q_v^{-s}}$ converges absolutely in $\Re(s) > \delta + \frac 1j$.
\end{keyobs}

This is a special case where we can make our above-mentioned ``moral'' precise.  It says that 
a product of local factors over primes of degree $j$ will not have in a pole which is more than 
$\frac 1j$ to the right of a pole of any local factor.

\begin{proof}
Note
\[ \log \prod_{v \in S_j} \frac 1{1-\alpha_v q_v^{-s}} = \sum_v \sum_m \frac{\alpha_v^m}{mq_v^{sm_v}}
\le \sum_v \sum_m \frac 1{q_v^{(s-\delta)m}} \]
If we denote by $p_v$ the rational prime below $q_v$, then $q_v \ge p_v^j$ so the above is bounded
(absolutely) by
\[ \sum_v \sum_m \frac 1{p_v^{(s-\delta)jm}} \le \sum_m \frac 1{m^{(s-\delta)j}}, \]
which converges if $(\Ree(s)-\delta)j > 1$, i.e., if $\Ree(s) > \delta + \frac 1j$.
\end{proof}

Now by Luo--Rudnick--Sarnak, we know each local factor of $L(s, \pi \times \bar \pi)$ has no pole
in $\Ree(s) \ge 1- \frac 2{n^2+1}$, so the above observation tells us $L_S(s,\pi \times \bar \pi)$
has no pole in $\Ree(s) > 1 - \frac 2{n^2+1} + \frac 1p$, since all but a finite number (which do not
matter) of primes in $S$ have degree $p$.  Hence if $p \ge \frac{n^2+1}2$ we are done.

To deal with smaller $p$, Ramakrishnan uses Kummer theory to prove the following.

\begin{lemma} Suppose $K/F$ is a degree $p^{m-1}$ extension such that $K/k$ is a nested chain of cyclic $p^2$-extensions.  If $v$ is a prime of $F$ of degree $p$ over $k$, and $w$ is an unramified
prime of $K$ over $v$, then $w$ has degree $p^m$ over $k$.
\label{lem:din}
\end{lemma}

Consequently, given such an extension $K/F$, the Key Observation tells us that
$L_S(s, \pi_{K} \times \bar \pi_{K})$ has no poles in $\Ree(s) > 1- \frac 2{n^2+1} + \frac 1{p^m}$.
Taking $m$ large enough we can conclude $L_S(s, \pi_{K} \times \bar \pi_{K})$ has no poles
in $\Ree(s) \ge 1$, and thus that $\pi_K \simeq \pi'_K$.  

To finish the proof, one must carefully vary the field $K$ to get $\pi_K \simeq \pi'_K$ over sufficiently
many extension $K/F$ to deduce the isomorphism $\pi \simeq \pi'$ over $F$.
\end{proof}

We will a use similar idea for our result in the next section.

\section{Modularity}

Let $\rho$ be an irreducible $n$-dimensional Artin representation of $\Gamma_F = \Gal(\bar F/F)$.  
Let $\pi$ be a cuspidal automorphic representation of $\GL_n(\A_F)$.  Here we are
interested in comparing $\rho$ and $\pi$.

Recall we say $\rho$ is modular if $L(s,\rho)$ agrees with $L(s,\pi(\rho))$ at almost all places, for
some cuspidal automorphic representation $\pi(\rho)$ of $\GL_n(\A_F)$.  (By SMO, $\pi(\rho)$
is unique up to isomorphism.)  The strong Artin, or modularity, conjecture asserts that
every $\rho$ is modular.  The following well-known result tells us an equivalent definition of
modularity is $L(s,\rho) = L(s,\pi)$ (in the sense of equality of Euler products).

\begin{prop} If $L(s,\rho_v) = L(s,\pi_v)$ for almost all $v$, then $L(s, \rho) = L(s,\pi)$.  Further
we have an identity of total archimedean factors $L_\infty(s,\rho) = L_\infty(s,\pi)$.
\end{prop}

The proof follows from an argument due to Deligne and Serre \cite{DS}, and the details are
given in \cite[Appendix A]{martin}.  The idea is to twist by a highly ramified character $\chi$ 
at bad places, which makes the $L$-factors 1 at these places so we get a global
equality $L(s, \rho \otimes \chi) = L(s, \pi \otimes \chi)$.  We may take $\chi$ to be trivial at each
archimedean place, and then comparing poles in functional equations allows us to deduce
$L_\infty(s,\rho) = L_\infty(s, \pi)$.  Now we repeat the argument with $\chi$ which is highly ramified
at all but one bad place $v$, where $\chi$ is trivial.  This gives $L(s, \rho_v) = L(s,\pi_v)$.

In fact, in \cite{MR}, when $n=2$ we show the stronger statement that $\rho$ and $\pi$ correspond
via local Langlands at all (finite and infinite) places.  However, this argument relies on the fact that 
the local Langlands correspondence is characterized by twists of $L$- and $\epsilon$- factors
by characters, which is not true for $n \ge 4$ (see
\cite[Remark 7.5.4]{JPSS:GL3} for an example with $n=4$),  so this argument does not generalize
to arbitrary $n$.

\medskip
Now we can ask, to compare $\rho$ and $\pi$, how large of a set of places do we need to deduce
$L(s,\rho) = L(s, \pi)$?
The above proposition says it suffices to compare them at almost all places.
But since $\rho$ and $\pi$ should be determined by their local $L$-factors outside any set $S$
of places of density less then $\frac 1{2n^2}$, to show they correspond it should suffice to
check matching of local $L$-factors outside such a set $S$.  Precisely, we have

\begin{conj} If $L(s,\rho_v) = L(s,\pi_v)$ for all $v$ outside of some set $S$ of places of
density less than $\frac 1{2n^2}$, then $L(s, \rho) = L(s, \pi)$.
\label{conj:new}
\end{conj}

This is true for $n=1$ by class field theory.  In general, this is a consequence of the
strong Artin conjecture together with the refined SMO Conjecture (Conjecture \ref{conj:din}).
Namely, if $\rho$ is modular and its $L$-function agrees with that of $\pi$ outside of $S$, then
$L(s,\pi_v) = L(s, \pi(\rho)_v)$ for $v \not \in S$.  By Conjecture \ref{conj:din}, if 
$\delta(S) < \frac 1{n^2}$, then $\pi \simeq \pi(\rho)$.

Consequently, by Theorem \ref{thm:dinSMO}, we know this conjecture is true whenever $n=2$
and $\rho$ is modular.  This is the case if $\rho$ has solvable image by the 
Langlands--Tunnell theorem, or if $\rho$ is odd and $F= \Q$ by Khare and Winterberger's work
on Serre's conjecture (see \cite{khare}).
Arguing similarly in the reverse direction, Conjecture \ref{conj:new} is true whenever $\pi$
corresponds to some Artin representation $\rho(\pi)$ by Proposition \ref{prop:gal-SMO}.
This is known if $\pi$ corresponds to a weight 1 Hilbert modular form by Wiles \cite{wiles}.
However, even for $n=2$, this is not solved completely, and the case where $\rho$ is even (so
$\pi$ should correspond to a Maass form, say, if $F=\Q$) with nonsolvable image seems particularly difficult.

In any case, one might hope that if one could prove this conjecture independent of modularity,
then this may help establish new cases of modularity.  Recently, Ramakrishnan and I proved
the following mild result towards this conjecture.

\begin{theorem}[\cite{MR}] Suppose $n=2$ and $F$ is cyclic extension of prime degree $p$ of some number field $k$.  Let $S \subset \Sigma_F$ be a set of primes such that almost all $v \in S$
are inert over $k$.  Then $L(s, \rho_v) = L(s, \pi_v)$ for $v \not \in S$ implies $L(s, \rho) = L(s, \pi)$,
and in fact $\rho_v \leftrightarrow \pi_v$ in the sense of local Langlands at all $v$.
\end{theorem}

\begin{proof}[Proof] By the generalization of the Deligne--Serre argument we mentioned above,
it suffices to show $L(s, \rho_v) = L(s, \pi_v)$ for almost all $v$.
We show this in 4 steps.  

\medskip
{\it Step 1:} Show $\pi$ is tempered (at each place).  

\medskip
It is immediate from the equality $L(s, \rho_v) = L(s, \pi_v)$ that $\pi_v$ is tempered for any
$v \not \in S$, so we just need to show temperedness at $v \in S$.
Note, from the Fact before Theorem \ref{thm:SMO}, 
$\pi_v$ is tempered if and only if $L(s, \pi_v \times \bar \pi_v)$
has no pole in $\Ree(s) > 0$.  
  Now consider the ratio
\[ \Lambda(s) = \Lambda_F(s) = \frac{L^*(s, \pi \times \bar \pi)}{L^*(s, \rho \times \bar \rho)} =
\frac{L_S(s, \pi \times \bar \pi)}{L_S(s, \rho \times \bar \rho)}. \]
Then $\Lambda(s)$ satisfies a functional equation with $\Lambda(1-s)$, and if we can show
$\Lambda$ has no poles in $\Ree(s) \ge \frac 12$, 
this will mean it is entire by the functional equation.

Take $\delta < \frac 14$ such that the bound of $q_v^{\delta}$ towards GRC is satisfied.  The Key Observation implies that $L_S(\pi \times \bar \pi)$ has no poles in $\Ree(s) > 2\delta + \frac 1p$
and $L_S(\rho \times \bar \rho)$ has no poles (and thus no zeroes by Landau's lemma) in 
$\Ree(s) > \frac 1p$.  For $p$ sufficiently large, this means $\Lambda$ has no poles in 
$\Ree(s) \ge \frac 12$, so $\Lambda$ is entire.  For small $p$, we use Lemma \ref{lem:din}
to pass to an extension $K$ to push our bounds on the poles of the numerator and denominator
to the left of $\Ree(s) = \frac 12$, and get that an analogous ratio $\Lambda_K$ is entire.
In either case, write the ratio as $\Lambda_K$, where we take $K=F$ if $p$ is sufficiently large.

If $\pi_v$ is not tempered for some $v \in S$, then Landau's lemma implies 
$L_{T}(s, \pi_K \times \bar \pi_K)$ has a pole
at some $s_0 > 0$, where $T$ is the set of primes of $K$ above $S$. 
But for $\Lambda_K$ to be entire, we also need a pole at $s_0$ for
 $L(s, \rho_K \times \bar \rho_K)$.  By taking $K$ larger if needed, we can make it so that
 $L(s, \rho_K \times \bar \rho_K)$ has no poles in $\Ree(s) > \frac 1{p^m} < s_0$, a contradiction.
 
\medskip
{\it Step 2:} Show $L(s, \rho_K)$, for some finite solvable 
extension $K/F$, is entire.

\medskip
Here we consider the ratio
\[ \Lambda_K(s) = \frac{L^*(s, \rho_K)}{L^*(s, \pi_K)} = \frac{L_T(s, \rho_K)}{L_T(s, \pi_K)}, \]
where as before $T$ is the set of places of $K$ above $S$.
Again, by looking at a functional equation, it suffices to show $L_T(s, \rho_K)$ has no pole in
 $\Ree(s) \ge \frac 12$.  The bound from the Key Observation is that $L_S(s, \rho)$ has no
 pole in $\Ree(s) > \frac 1p$, so can take $K=F$ unless $p=2$.  In this case we need to pass
 to an extension $K$, so the bound becomes $\Ree(s) > \frac 1{p^2}$.  By a refinement of 
 Lemma \ref{lem:din}, we can do this with $K/F$ quadratic or biquadratic, according to whether
 $\sqrt{-1} \in F$ or not.

\medskip
{\it Step 3:} Deduce $L(s, \rho_K) = L(s, \pi_K)$.

\medskip
The point is, up until now, everything we did is valid for twists, and the choice of $K$ in the
previous step only depends on $F/k$.  So we get that $L(s, \rho_K \otimes \chi)$ is entire for
any finite order idele class character of $K$.  For Artin representations, it is known that this is
sufficient to use the $\GL(2)$ converse theorem, namely one gets boundedness in vertical strips
for free.  Thus $\rho_K$ corresponds to an automorphic representation $\Pi$ of $\GL_2(\A_K)$,
which must have the same $L$-factors as $\pi_K$ at all places not above a place in $S$, i.e.,
at all places outside a density 0 set.  By refined SMO (Theorem \ref{thm:dinSMO}), this means
$\pi_K \simeq \Pi$, so $L(s, \rho_K) = L(s, \pi_K)$.

\medskip
{\it Step 4:} Descend the previous step to $F$, i.e., show $L(s, \rho) = L(s, \pi)$.

\medskip
If $p>2$, then $K=F$ so there is nothing to do.  Assume $p=2$.  I will just discuss the proof
in the simpler case that $K/F$ is quadratic (i.e., when $\sqrt{-1} \in F$), 
and refer to \cite{MR} for the biquadratic case.
Remember, it suffices to show for almost all $v \in S$ that
$L(s, \rho_v) = L(s, \pi_v)$.  Fix any place $v \in S$ such that $\rho_v$ and $\pi_v$ are
unramified, and let $w$ be a place above $v$, which will be inert.  By the previous step, we know
$\rho_{K,w} \leftrightarrow \pi_{K,w}$ (in the sense of local Langlands, since unramified 
representations are determined by their local $L$-factors).  
This means that $\rho_v$ must correspond to either $\pi_v$ or $\pi_v \otimes \mu$, 
where $\mu$ is the quadratic character associated $K_w/F_v$.  Now the point is that we have
sufficient flexibility in our choice of $K$ so that we can get the correspondence 
$\rho_{K,w} \leftrightarrow \pi_{K,w}$ both when $K_w/F_v$ is ramified.
But it is impossible for an unramified $\rho_v$ to correspond to a ramified twist of the
unramified $\pi_v$, so we must have $\rho_v \leftrightarrow \pi_v$, and we are done.
\end{proof}

Finally we remark that a similar conjecture should be true for (families of compatible)
$\ell$-adic Galois representations.  In order for the proof of the above theorem to go through for
$\ell$-adic Galois representations, first we would need to know a purity result (which is known
in some cases), such as $L(s,\rho_v)$
has no poles in $\Ree(s) > 0$, to conclude temperedness of $\pi$.  Then we would need to know
that entirety of the twists $L(s, \rho \otimes \chi)$ also implies boundedness in vertical strips, so we
can use the converse theorem in Step 3.  (For Artin representations, this follows from a theorem of Brauer
which tells us Artin $L$-functions are quotients of products of degree 1 $L$-functions.)

\end{document}